\documentclass[reqno]{amsart}
\usepackage{amsfonts,amsmath,amsthm,amssymb,latexsym, cite}
\usepackage{graphicx}
\usepackage[colorlinks,plainpages,citecolor=red, linkcolor=blue,bookmarksnumbered]{hyperref}

\theoremstyle{plain}
\newtheorem{thms}{Theorem}
\newtheorem{thm}{Theorem}[section]
\newtheorem{defn}[thm]{Definition}

\newtheorem{lem}[thm]{Lemma}
\newtheorem{rem}[thm]{Remark}

\newtheorem{prop}[thm]{Proposition}

\numberwithin{equation}{section}
\newcommand{\dom}{\mathop{\rm dom}}
\renewcommand{\Re}{\mathop{\rm Re}}
\renewcommand{\Im}{\mathop{\rm Im}}
\newcommand{\supp}{\mathop{\rm supp}}

\renewcommand{\kappa}{\varkappa}

\newcommand{\Real}{\mathbb R}
\newcommand{\Comp}{\mathbb C}

\newcommand{\eps}{\varepsilon}

\newcommand{\cI}{\mathcal{I}}
\newcommand{\fpr}[1]{{#1}^{(-1)}}
\newcommand{\spr}[1]{{#1}^{(-2)}}
\newcommand{\ra}{\rangle}
\newcommand{\la}{\langle}

\renewcommand{\phi}{\varphi}

\newcommand\xe{\left(\tfrac x\eps\right)}

\newcommand\se{\left(\tfrac s\eps\right)}

\newcommand{\scrp}{\scriptstyle}

\usepackage {tikz}
\usetikzlibrary {positioning}
\definecolor {processblue}{cmyk}{0.96,0,0,0}
\newcommand{\Graph}{
\begin {tikzpicture}[-latex , auto , node distance =2 cm and 3.1 cm ,on grid, thin,
state/.style ={circle, top color =white, bottom color = processblue!20,
draw, processblue, text=black, minimum width =0.3 cm}]
\node[state] (T) {$H_\eps$};
\node[state] (B3) [below left =of T] {B3};
\node[state] (A2) [below=of B3]        {A2};
\node[state] (B2) [below  = of A2] {B2};
\node[state] (B1) [below  =of B2] {B1};
\node[state] (A1)[right =of B1] {A1};
\node[state] (E20) [right  =of A1] {};
\node[state] (N2)[below =of T] {};
\node[state] (N1) [right=of N2] {};
\node[state] (E1) [  right =of   N1] {};
\node[state] (X) [above =of   E1] {x};
\node[state] (E11) [below =of  E1] {};
\node[state] (A3) [below =of E11] {A3};
\node[state] (B31)[ below   =of A3]  {B3};
\node[state] (N3)[ right=of A2] {};
\node[state] (E21) [right =of B2]  {};
\node[state] (E2) [right =of  N3] {};
\node[state] (E22) [ right  =of   E21] {};

\path (T)   edge [bend right=10]    node[ right=0.2 cm]
  {$\scrp\lambda\neq 0$}    (B3);

\path (T)   edge [bend left =10]   node[ left  =0.4 cm]
 {$\scrp\lambda=0$}        (E1);

\path (E1)  edge [bend right =10]  node[above left =0.0 cm]  {$ |f_0|+|g_0|\neq0,\atop f_0 g_0=0$}    (X);

\path (E1)  edge [bend left =20]  node[above  =0 cm]
{$\scrp f_0g_0\neq 0$}    (N1);

\path (E1)  edge [bend right =15]  node[left =0 cm]
{$f_0=0,\atop g_0=0$}    (E11);

\path (N1) edge [bend right =10]  node[above  =0. cm]
{$\scrp f_0g_1=f_1g_0$}   (N2);

\path (N2) edge [bend left =10]  node[above  =0. cm]
{$\scrp \sigma_+=0$}   (B3);

\path (N2) edge [bend left =10]  node[above left =0. cm]
{$\scrp \sigma_+\neq0$}   (A2);

\path (N1) edge [bend left =15]   node[above left =-0.1 cm]
{$\scrp f_0g_1\neq f_1g_0$}      (N3);

\path (N3) edge [bend left =10]  node[above =0. cm]
{$\scrp \sigma_-\sigma_+ \neq0$}   (A2);

\path (N3) edge [bend left =10]  node[left =0. cm]
{$\scrp \sigma_-\sigma_+ =0$}   (B2);

\path (E11) edge [bend left=10]   node[ left =0.0 cm]
{$\scrp \pi \neq 0$}    (A3);

\path (E11) edge [bend right =20]  node[below =0. cm]
 {$\scrp \pi=0$}        (E2);

\path (E2)  edge [bend left=10]   node[ below =0.15 cm]
  {$\scrp \kappa\neq 0$}      (E21);

\path (E2)  edge [bend left =10]   node[ right =0. cm]
        {$\scrp \kappa=0$}         (E22);

\path (E21) edge [bend right=10]   node[ below   =0.3 cm]
   {$\scrp a_2=\overline{\kappa} a_1$}(B1);

\path (E21) edge [bend right =20]   node[  right   =-0.05 cm]
   {$\scrp a_2\neq \overline{\kappa} a_1$}    (A1);

\path (E22) edge [bend right=10]    node[below right =-0.1 cm]
    {$\scrp a_2\neq 0$}           (A1);

\path (E22) edge [bend right =15]  node[right=0. cm]
        {$\scrp a_2=0$}               (E20);

\path (E20)  edge [bend left =10]  node[  right   =0.15 cm]
  {$\scrp a_1=0$}               (A3);

\path (E20)  edge [bend right =15]   node[above =0. cm]
   {$\scrp a_1\neq 0$}           (B31);
\end{tikzpicture}
}

\begin{document}

\title[Schr\"{o}dinger operators with singular rank-two perturbations]
{Schr\"{o}dinger operators with singular rank-two perturbations and  point interactions}

\author{Yuriy Golovaty}%
\address{Department of Mechanics and Mathematics,
  Ivan Franko National University of Lviv\\
  1 Universytetska str., 79000 Lviv, Ukraine}
\curraddr{}
\email{yu\_\,holovaty@franko.lviv.ua}

\subjclass[2000]{Primary 34L40, 34B09; Secondary  81Q10}

\begin{abstract}
  Norm resolvent approximation for a wide class of point interactions in one dimension is constructed. To analyse the limit behaviour of  Schr\"{o}dinger operators with localized singular rank-two perturbations coupled with $\delta$-like potentials as the support of perturbation shrinks to a point, we show that the set of  limit operators is quite rich. Depending on parameters of the perturbation, the limit operators are described by both the connected and separated boundary conditions. In particular  an approximation for a four-parametric subfamily of all the connected point interactions is built.
 We give examples of the singular perturbed Schr\"{o}dinger operators without loca\-li\-zed gauge fields, which converge to point interactions with the non-trivial phase parameter. We also construct an approximation for the point interactions that are described by different types of the separated boundary conditions such as the Robin-Dirichlet, the Neumann-Neumann or the Robin-Robin types.
\end{abstract}

\keywords{1d Schr\"{o}dinger operator, point interaction, $\delta'$-potential, $\delta'$-interaction, solvable model, finite rank perturbation, scattering problem}
\maketitle


\section{Introduction}

Solvable Schr\"odinger type operators  have attracted considerable attention both in the physical and mathematical literature in recent years. Such the operators are of interest in applications of mathematics in different fields of science and engineering. The so-called solvable models that are based upon the concept of point interactions also often appear in quantum theory and allow us to calculate explicitly spectral characteristics of systems such as eigenvalues, eigenfunctions or scattering data. The  Schr\"odinger operators  with singular distributional potentials supported on  discrete sets  reveal an unquestioned effectiveness whenever the exact solvabi\-li\-ty together with non trivial description of the actual process is required.
It is an extensive subject with a large literature, see the books by Albeverio, Gesztesy, H{\o}egh-Krohn, and Holden~\cite{Albeverio2edition} and  Albeverio and Kurasov~\cite{AlbeverioKurasov} discussing point interactions and more general singular perturbations of the  Schr\"odinger operators and the extensive bibliography lists therein.


In spite of all advantages of the solvable models, they  give rise to many ma\-the\-matical difficulties. One of the main difficulty deals with the problem of defining a  multiplication of distributions.
It entails that many Schr\"{o}dinger operators with singular potentials are often only formal differential expressions without a precise mathematical meaning. We cite two linear differential equations with distributions contained in the coefficients  as an example.

First let us consider the  model of the  Schr\"{o}dinger equation $-y''+\delta(x)y=k^2y$ with the $\delta$ potential. Here $\delta$ is the Dirac delta-function and $\delta(x)y(x)=y(0)\delta(x)$. It can be also written in the form $-y''+\langle\delta(x),y\rangle\,\delta(x)=k^2y$.
It is well-known that both the equations have the same $2$-dimensional space of solutions in the sense of distributions. All the solutions are continuous  at the origin and therefore the product $\delta(x)y(x)$ and the value  $\langle\delta(x),y\rangle$ are well-defined. Both the  differential expressions $-\frac{ d^2}{ d x^2} + \delta(x)$ and $-\frac{ d^2}{ d x^2} + \langle\delta(x),\,\cdot\,\rangle\,\delta(x)$ could be interpreted as the same self-adjoin operator in $L_2(\Real)$, defined by $Sy=-y''$
on functions  in $W^2_2(\Real\setminus\{0\})$ obeying the interface conditions $y(-0)=y(+0)$, $y'(+0)-y'(-0)=y(0)$.

At the same time, the  equation $-y''+\delta'(x)y=k^2y$
with the first derivative of the Dirac delta-function as a potential has no  mathe\-ma\-ti\-cal sense, because for it no solution exists in the space of distributions, except the trivial one.
Indeed, the product $\delta'(x) y=y(0)\delta'(x)-y'(0)\delta(x)$ is well defined for  $y$ that is continuously differentiable at the origin. But this is impossible for a non-trivial solution, because its second derivative  is the singular distribution $y(0)\delta'(x)-y'(0)\delta(x)+k^2y$. The equation
$
-y''+\langle \delta(x),y\rangle\, \delta'(x)+\langle \delta'(x),y\rangle\, \delta(x)=k^2y
$, in which potential $\delta'(x)$ is treated as the rank-two perturbation, is also meaningless.

Hence the situation is more obscure with definition of the Schr\"odinger operators with  potential $\delta'$, and one must be careful in using the formal differential expressions
\begin{align}\label{PseudoH1}
  -&\frac{ d^2}{ d x^2} + \alpha\delta'(x)+ \beta\delta(x), \\\label{PseudoH2}
  -&\frac{d^2}{dx^2}+\alpha\big(\langle \delta'(x),\,\cdot\,\rangle \,\delta(x)
  +\langle \delta(x),\,\cdot\,\rangle\, \delta'(x)\big)+\beta \delta(x).
\end{align}
However, such the pseudo-Hamiltonians often appear in the models of quantum devices with  barrier-well junctions. To get round  the problem  of multiplication of distributions, we can regularize $\delta'$  by smooth enough localized potentials  and then investigate  convergence of the Schr\"odinger operators with the regular potentials.
The main goal is to find the limit operator  and  assign for the quantum system a solvable model (i.e., a point interaction) that governs the  quantum process of the true interaction with adequate accuracy. Note that such the results depend on shapes of the approximation sequences. From a physical point of view, this means that there are many different ``$\delta'$ potentials'', namely, the quantum devices with  $\delta'$-like potentials of various  shapes exhibit the different properties.

The Schr\"{o}dinger operators with $(\alpha\delta'+\beta\delta)$-like potentials, i.e., the regularization of the pseudo-Hamiltonian   \eqref{PseudoH1}, was studied in
\cite{GolovatyMankoUMB, GolovatyHrynivJPA2010, GolovatyMFAT2012, GolovatyHrynivProcEdinburgh2013,GolovatyIEOT2013}.  The norm resolvent convergence of the corresponding families of operators  was established and a class of solvable models that
approximate the quantum systems with barrier-well junctions was  obtained. The result of \cite{SebRMP} about the regularization of $\delta'$-potential was revised and adjusted in \cite{GolovatyHrynivJPA2010}.
Different families  of Schr\"{o}dinger operators with potentials of the dipole type using a regularization by rectangles in the form of a barrier and a well were treated by Zolotaryuk (partly  with coauthors) in   \cite{ChristianZolotarIermak03, Zolotaryuk08,Zolotaryuk10, Zolotaryuks11, ZolotaryukThreeDelta17}.

In this paper we study families of Schr\"{o}dinger operators with localized singular rank-two perturbations coupled with $\delta$-like potentials. These operators can be regarded as the regularization of the pseudo-Hamiltonian  \eqref{PseudoH2}, but only in a special case.
A careful analysis actually shows that the families  describe
a variety of quantum interactions and the set of all limit operators, which can be obtained in the norm resolvent topology as the support of perturbation shrinks to the origin, contains a wide class of point interactions. The limit operators are described by both the connected and separated boundary conditions. In the first case, we obtained the approximation for a four-parametric subfamily of all the connected point interactions with a complete matrix in the boundary conditions.
Moreover an unexpected fact is that the point interactions with non-trivial phase parameter appear in the limit, although   the perturbed Schr\"{o}dinger operators contain no localized magnetic field.
We also constructed an approximation for point interactions that are described by different types of the separated boundary conditions such as the Robin-Dirichlet, the Neumann-Neumann or the Robin-Robin types.
A partial case of the problem has been recently published in \cite{Golovaty2018}.

Problems of this nature have a long history and the literature
on approximation for point interactions as well as finite rank perturbations of the Schr\"{o}dinger ope\-ra\-tors is extensive.  Among all zero range  interactions, the $\delta'$-interactions, along with $\delta$ and $\delta'$ potentials, are most studied in this kind of research. We want to especially note the paper \cite{AlbeverioNizhnik2000,CheonShigehara1998, CheonExner2004,  ExnerManko2014, ExnerNeidhardtZagrebnov2001, NizhFAA2003, SebRMP} and the references therein. This special case has attracted much attention recently
\cite{AlbeverioFassariRinaldi2013, AlbeverioFassariRinaldi2015, Lange2015, ZolotaryukThreeDelta17}.
Many authors have dealt with finite rank perturbations and their relationship with the point interactions. In particular we mention  papers on singular finite rank perturbations and nonlocal potentials
\cite{AlbeverioKoshmanenkoKurasovNizhnik2002, AlbeverioNizhnik2000, AlbeverioNizhnik2007, AlbeverioNizhnik2013, HassiKuzhel2009, KuzhelZnojil, Nizhnik2001, KuzhelNizhnik2006, KoshmanenkoUMZh1991}.

\section{Statement of Problem, Main Results and Discussion}
From now on, the scalar product and norm in $L_2(\Real)$ will be denoted by $\la\cdot,\cdot\ra$ and $\|\cdot\|$ respectively.
Let us consider the Schr\"{o}dinger operator
\begin{equation*}
  H_0=-\frac{d^2}{dx^2}+V(x)
\end{equation*}
in $L_2(\Real)$,  where potential $V$ is  real-valued, measurable and locally bounded. We also assume that $V$ is bounded from below in $\Real$. Let $f$ and $g$ be complex-valued and compactly supported functions in $L_2(\Real)$ that are linearly independent.
We  denote by $Q_\eps$ the rank-two operators
\begin{multline*}
  (Q_\eps v)(x)=\la g(\eps^{-1}\,\cdot),v\ra \,f(\eps^{-1}x)
  +\la f(\eps^{-1}\,\cdot),v\ra\, g(\eps^{-1}x)\\
  =\int_\Real \left(\bar{g}(\eps^{-1}s)f(\eps^{-1}x)
  +\bar{f}(\eps^{-1}s)g(\eps^{-1}x)\right)v(s)\,ds
\end{multline*}
acting in $L_2(\Real)$. Let us  consider the family of self-adjoint operators
\begin{equation*}
 H_\eps= H_0+\eps^{-3}Q_\eps+\eps^{-1}q(\eps^{-1} x),
\end{equation*}
where $q$ is an integrable real-valued bounded function of compact support. Since the perturbation of $H_0$ has a compact support, we have  $\dom H_\eps=\dom H_0$.

One of the questions of our primary interest in this paper is to understand the limiting behavior of the operators $H_\eps$ as the small positive parameter $\eps$ goes to zero, i.e., as  the support of perturbation shrinks to the origin.
An asymptotic  analysis of $H_\eps$ leads us to a few cases of norm resolvent limits.  This limiting behaviour is governed primarily by  $f$ and $g$ as well as their interaction with the potential $q$.

We introduce  notation
\begin{equation*}
  f_0=\int_\Real f\,dx,\quad g_0=\int_\Real g\,dx,\quad f_1=\int_\Real x  f\,dx,\quad g_1=\int_\Real x g\,dx.
\end{equation*}
Let us denote by $h^{(-1)}(x)=\int_{-\infty}^x h(s)\, ds$ and $h^{(-2)}(x)=\int_{-\infty}^x (x-s)h(s)\, ds$
the first and second  antiderivatives of a function $h$. The antiderivatives are well-defined for measurable functions of compact support, for instance. In addition, if $h$ has zero mean, then $h^{(-1)}$ is also a function of compact support. We will henceforth use notation
\begin{gather}
\label{Pi}
    \pi=
    \|\fpr{f}\|\cdot\|\fpr{g}\|-|\la \fpr{f},\fpr{g}\ra+1|,
\\\nonumber
    \omega=e^{i\arg(\la \fpr{f},\fpr{g}\ra+1)}\|\fpr{g}\|\, \spr{f}-\|\fpr{f}\|\, \spr{g},
\\\label{Ak}
    a_0=\int_\Real q\,dx,\quad a_1=\int_\Real q\,\omega\,dx,\quad a_2=\int_\Real q\,|\omega|^2\,dx.
\end{gather}
that will be correct only if   $f$ and $g$ have zero means, i.e., $f_0=0$ and $g_0=0$. In this case, $\omega$ is  constant outside  some interval that contains the supports of $f$ and $g$. We write
\begin{equation*}
  \kappa=\lim_{x\to+\infty}\omega(x).
\end{equation*}
Of course,  $\lim_{x\to-\infty}\omega(x)=0$.
We  also set
\begin{gather}
\nonumber
\lambda=\|g_0\fpr{f}-f_0 \fpr{g}\|^2-2\Re{(f_0\bar{g}_0)},
\\\nonumber
  \sigma=|g_0|^2\left(\bar{f}_0\spr{f}-\la f,\spr{f}\ra\right)-
       |f_0|^2\left(\bar{g}_0\spr{g}-\la g,\spr{g}\ra\right),
 \\\label{sigmapm}
       \sigma_-=\lim_{x\to-\infty}\sigma(x), \quad \sigma_+=\lim_{x\to+\infty}\sigma(x), \quad
       \sigma_*=\int_\Real q|\sigma|^2\,dx.
\end{gather}
Remark that  $g_0\fpr{f}-f_0 \fpr{g}$ belongs to $L_2(\Real)$, because $g_0f-f_0g$ is a function of zero mean. The limits $\sigma_-$ and $\sigma_+$ also exist (to be proved later in Lemma~\ref{LemmaHBS}).

Let us introduce the subspace $\mathcal{V}\subset L_2(\Real)$ as follows.
We say that $h$ belongs to $\mathcal{V}$ if there exist two functions $h_-$ and $h_+$ belonging to $\dom H_0$ such that $h(x)=h_-(x)$ if $x<0$ and $h(x)=h_+(x)$ if $x>0$.
Let $(c_{ij})$ be a square matrix of order $2$ with real elements and $\phi\in\Real$.
We denote by $\mathcal{H}$ the operator defined by
\begin{equation*}
\mathcal{H}\,v=-\frac{d^2 v}{dx^2}+V(x)v
\end{equation*}
on functions $v\in \mathcal{V}$  obeying the coupling conditions
\begin{equation}\label{ConnectedCond}
 \begin{pmatrix} v_+ \\ v'_+ \end{pmatrix}
= e^{i\phi}
\begin{pmatrix}
  c_{11} &  c_{12}\\
  c_{21} &  c_{22}
\end{pmatrix}
\begin{pmatrix} v_- \\ v'_- \end{pmatrix}.
\end{equation}
at the origin. Here $v_-=v(-0)$, $v_+=v(+0)$ and $i^2=-1$.  Remark that  $\mathcal{H}$ is self-adjoint if and only if $c_{11}c_{22}-c_{12}c_{21}=1$.

The following theorem collects the cases  of  limiting behaviour of $H_\eps$ in which the limit operators describe non-trivial point interactions; these cases are of special interest in the scattering theory. Under a non-trivial point interaction we understand
the point interaction that describes by boundary conditions \eqref{ConnectedCond}.

\begin{thms}\label{ThmC}
Let $f, g\colon \Real\to\Comp$ and $q\colon \Real\to\Real$ be integrable functions of compact support, and $f$ and $g$ are linearly independent in $L_2(\Real)$.
\begin{itemize}
\setlength\itemsep{1em}
  \item[A1.] If $f_0=0$, $g_0=0$, $\pi=0$ and $a_2\neq\overline{\kappa} a_1$, then operators
$H_\eps$  converge in the norm resolvent sense as $\eps\to 0$  to the operator $\mathcal{H}$ with   coupling conditions
\begin{equation}\label{MatrixHinC1}
  \begin{pmatrix} v_+ \\ v'_+ \end{pmatrix}
=
e^{i \arg (a_2-\kappa \bar{a}_1)}\begin{pmatrix}
\dfrac{|\kappa|^2a_0-2{\rm Re}(\overline{\kappa} a_1)+a_2}
{|a_2-\overline{\kappa} a_1|}    &    \dfrac{|\kappa|^2}{|a_2-\overline{\kappa} a_1|}\\
      \dfrac{a_0a_2-|a_1|^2}{|a_2-\overline{\kappa} a_1|}  &      \dfrac{a_2}{|a_2-\overline{\kappa} a_1|}
  \end{pmatrix}
  \begin{pmatrix} v_- \\ v'_- \end{pmatrix}.
\end{equation}
  \item[A2.] Suppose  $\lambda=0$, $f_0g_0\neq 0$ and $\sigma_- \sigma_+\neq 0$,
      then  $\sigma_+$ is a real number and $H_\eps$ converge  to operator $\mathcal{H}$  in the norm resolvent sense, where $\dom\mathcal{H}$ consists of functions $v\in\mathcal{V}$ such that
      \begin{equation}\label{MatrixHinA2}
  \begin{pmatrix} v_+ \\ v'_+ \end{pmatrix}
=
e^{-i \arg \sigma_-}\begin{pmatrix}
\dfrac{\sigma_+}{|\sigma_-|}    &    0\\
      \dfrac{\sigma_*}{\sigma_+|\sigma_-|}  &      \dfrac{|\sigma_-|}{\sigma_+}
  \end{pmatrix}
  \begin{pmatrix} v_- \\ v'_- \end{pmatrix}.
\end{equation}

\item[A3.] Assume  $f_0=0$, $g_0=0$ and one of the following conditions holds
   \begin{itemize}
   \item[$\circ$] $\pi\neq 0$;
   \item[$\circ$] $\pi=0$,  $\kappa=0$, $a_1=0$ and $a_2=0$.
 \end{itemize}
Then  resolvents of $H_\eps$ converge in norm as $\eps\to 0$ to the resolvent of opera\-tor $\mathcal{H}$ with  coupling conditions
  \begin{equation}\label{MatrixHinC2}
   \begin{pmatrix} v_+ \\ v'_+ \end{pmatrix}
   =
   \begin{pmatrix}
    1     &    0\\
    a_0 & 1
  \end{pmatrix}
  \begin{pmatrix} v_- \\ v'_- \end{pmatrix}.
\end{equation}
\end{itemize}
\smallskip
\end{thms}

Remark that function $g_0\fpr{f}-f_0 \fpr{g}$ belongs to $L_2(\Real)$, because $g_0f-f_0g$ has the zero mean.

Let $\mathcal{V}_-$ and $\mathcal{V}_+$ be the spaces
obtained by the restriction of all elements of $\mathcal{V}$ to $\Real_-$ and $\Real_+$ respectively.
We introduce the operators
\begin{align*}
  \mathcal{D}_\pm=& -\frac{d^2}{d x^2}+V, \qquad \dom \mathcal{D}_\pm=\{v\in \mathcal{V}_\pm\colon v(0)=0\};  \\
  \mathcal{R}_\pm(\theta)=&-\frac{d^2}{d x^2}+V, \qquad \dom \mathcal{R}_\pm=\{v\in \mathcal{V}_\pm\colon v'(0)=\theta v(0)\},
\end{align*}
where $\theta\in\Real$.
In the next theorem we will assemble together all cases, when the limit operator is  a direct sum of two operators acting independently on the negative and positive semiaxes. The corresponding point interactions
are described by the boundary conditions
\begin{equation}\label{SepConds}
\alpha_1v_-+\beta_1v'_-=0, \quad \alpha_2v_++\beta_2v'_+=0
\end{equation}
with real coefficients $\alpha_k$ and $\beta_k$. These conditions are called separated in contrast to conditions \eqref{ConnectedCond}, which are called connected.

\begin{thms}\label{ThmS}
Let $f, g\colon \Real\to\Comp$ and $q\colon \Real\to\Real$ be integrable functions of compact support, and $f$ and $g$ are linearly independent in $L_2(\Real)$.
\begin{itemize}
\setlength\itemsep{1em}
\item[B1.] Suppose that $f_0=0$, $g_0=0$, $\pi=0$, $\kappa\neq 0$, and $a_2=\overline{\kappa} a_1$. Then  operators  $H_\eps$  converge to  the direct sum $\mathcal{R}_-(\theta_1)\oplus \mathcal{R}_+(\theta_2)$
      as $\eps\to 0$ in the norm resolvent sense, where
  $\theta_1=|\kappa|^{-2}a_2-a_0$ and $\theta_2=|\kappa|^{-2}a_2$.

\item[B2.] If $\lambda=0$, $f_0g_0\neq 0$, $f_1g_0\neq f_1g_0$ and $\sigma_- \sigma_+=0$, then
      \begin{equation*}
        H_\eps\to
        \begin{cases}
          \mathcal{D}_-\oplus\mathcal{R}_+(\theta_+), & \text{if } \sigma_-=0,\\
          \mathcal{R}_-(\theta_-)\oplus \mathcal{D}_+, &\text{if }\sigma_+=0,
        \end{cases}\qquad \text{as } \eps\to 0
      \end{equation*}
      in the norm resolvent sense, where $\theta_+=\sigma_*|\sigma_+|^{-2}$ and $\theta_-=-\sigma_*|\sigma_-|^{-2}$.
  \item[B3.]   Suppose that one of the following conditions holds
 \begin{itemize}
   \item[$\circ$] $\lambda\neq0$;
   \item[$\circ$] $\lambda=0$,   $f_0g_0\neq 0$, $f_0g_1=f_1g_0$, $\sigma_-=0$ and $\sigma_+=0$;
   \item[$\circ$] $f_0=0$, $g_0=0$, $\pi=0$, $\kappa=0$, $a_2=0$ and $a_1\neq0$.
 \end{itemize}
Then the operator family $H_\eps$  converges to  the direct sum $\mathcal{D}_-\oplus \mathcal{D}_+$  in the norm resolvent sense.
\end{itemize}
\end{thms}

As depicted in the  graph (Fig.~\ref{FigBranchingAlgorithm}), Theorems~\ref{ThmC} and \ref{ThmS} cover all limit cases as $\eps\to 0$. We need only note explicitly that if $f_0g_1=f_1g_0$, then $\sigma_-=\sigma_+$ (we will prove this fact below).
We also remark   that the case when $\|g_0\fpr{f}-f_0 \fpr{g}\|^2=2\Re{(f_0\bar{g}_0)}$,  $f_0g_0=0$, but only one of  mean values $f_0$ and $g_0$ equals zero (the node ``x'' of the graph), is impossible under our assumption about linear independence of $f$ and $g$. For instance, if $f_0=0$ and $g_0\neq 0$, then condition  $\|g_0\fpr f\|^2=0$ yields $f=0$.

Theorems~\ref{ThmC} and \ref{ThmS} can be summarized by saying that
the family of operators $H_\eps$ always  converges in the norm resolvent sense.

\begin{thms}\label{Thm3}
Let $f, g\colon \Real\to\Comp$ and $q\colon \Real\to\Real$ be integrable functions of compact support. Assume  $f$ and $g$ are linearly independent in $L_2(\Real)$ and define the sequences of scaled functions        $f^\eps=f(\eps^{-1}\,\cdot)$, $g^\eps=g(\eps^{-1}\,\cdot)$ and $q^\eps=q(\eps^{-1}\,\cdot)$. Then operator family $H_\eps=H_0+\eps^{-3}\la g^\eps,\,\cdot\,\ra \,f^\eps
  +\eps^{-3}\la f^\eps,\,\cdot\,\ra\, g^\eps+\eps^{-1}q^\eps$
converges in the norm resolvent sense as $\eps\to 0$ to some operator $\mathcal{H}=\mathcal{H}(f,g,q)$ and we have estimate
  \begin{equation}\label{MainEst}
  \|(H_\eps-\zeta)^{-1}-(\mathcal{H}-\zeta)^{-1}\|\leq C \eps^{1/2}
  \end{equation}
for all $\zeta\in \Comp\setminus\Real$, where the constant $C$ does not depend  on  $\eps$. The limit operator $\mathcal{H}$ is described in Theorems~\ref{ThmC} and \ref{ThmS}.
\end{thms}

\begin{figure}[b]
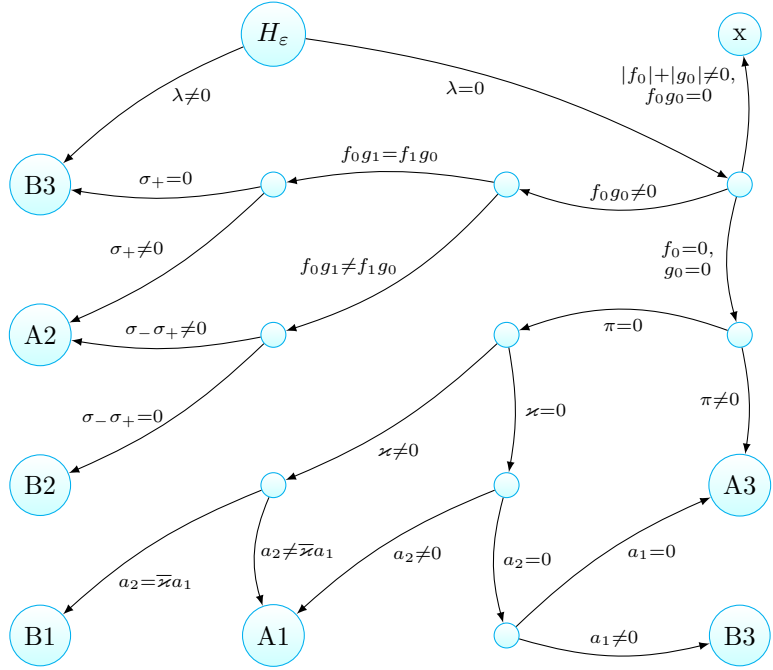

  \centering
 \Graph
 \\
  \caption{Bifurcation graph of limiting behaviour of $H_\eps$.}
  \label{FigBranchingAlgorithm}
 \end{figure}

Remark that the $\delta$-like sequence $\eps^{-1}q^\eps$ is obviously   subordinated to the rank-two perturbation as $\eps\to 0$, ne\-ver\-the\-less it  has a considerable influence on the limit behaviour of $H_\eps$. We should note that  the most interesting case A1 is possible only if the potential $q$ is different from zero.

All the cases A1--A3, B1--B3  can be realized by a proper choice of  triple $(f,g,q)$. For instance,  we explain how to choose the triple in case A1.
Let us consider two  functions $F$ and $G$ of compact support, belonging to $W_2^1(\Real)$, such that $\|F\|=\|G\|=1$ and $\la F,G\ra=0$. Then  $f=F'$ and $g=G'$ have zero means and satisfy the condition $\pi=0$. Next, $\omega=\fpr{F}-\fpr{G}$ and we can calculate $\kappa=\omega(+\infty)$. The linear independence of $F$ and $G$ implies
the linear independence of $f$ and $g$ and also the linear independence of functions $1$, $\omega$ and $\omega^2$ on each interval $[-r,r]$. Therefore for any  $(a_0,a_2)\in\Real^2$ and $a_1\in \Comp$ there exists a potential $q$ of compact support satisfying \eqref{Ak}. In particular, the potential can be chosen in such a way that  $a_2\neq\overline{\kappa} a_1$.

In this connection, the next question arises as to whether any real matrix $(c_{kl})$ with the unit determinant can be realized in coupling conditions \eqref{MatrixHinC1} for some  $f$, $g$ and $q$.
The answer is negative, the matrix
\begin{equation}\label{MatrixDpr}
   \begin{pmatrix}
     \alpha & 0  \\
     \beta    &  \alpha^{-1}
  \end{pmatrix}
\end{equation}
with $\alpha\neq 1$ is a counterexample. In fact, if we put $\kappa=0$ in \eqref{MatrixHinC1}, we immediately obtain the matrix with the unit diagonal. However, such matrices appear in  case~A2. It is worth  mentioning that the point interactions given by \eqref{MatrixDpr} also arise in  analysis of the  Schr\"{o}dinger operators with $(a\delta'+b\delta)$-like potentials \cite{GadellaNegroNietoPL2009, GadellaGlasserNieto2011, GolovatyMFAT2012, GolovatyIEOT2013, Zolotaryuks11}.

Now we will discuss a few special subcases.

\smallskip

\paragraph{\textit{Regularization of  pseudo-Hamiltonian \eqref{PseudoH2}.}}
If we suppose that $f_0=\alpha$, $g_0=0$ and $g_1=-1$, then
$\eps^{-1}f(\eps^{-1}x)\to \alpha\delta(x)$ and  $\eps^{-2}g(\eps^{-1}x)\to \delta'(x)$, as $\eps\to 0$, in the sense of distributions. Then the family $H_\eps$ can be treated as the regularization of the formal operator~\eqref{PseudoH2} with $\beta=a_0$.
Suppose that $\alpha$ is different from zero. Since $\lambda=\alpha^2\|\fpr{g}\|^2\neq 0$,
we fall into the conditions of  case B3, and so $H_\eps$ converge to  the direct sum $\mathcal{D}_-\oplus \mathcal{D}_+$ in the norm resolvent sense.

\smallskip

\paragraph{\textit{Generalized $\delta'$-interactions.}}
Suppose that $f$ and $g$ are real-valued functions.
Under the assumptions of case A1, we assume that  $a_0a_2=a_1^2$. Then operators $H_\eps$ give us the approximation to the point interactions with matrix
        \begin{equation*}
         \begin{pmatrix}
               \alpha &    \beta\\
                0      &  \alpha^{-1}
           \end{pmatrix},
        \end{equation*}
where $\alpha=a_2^{-1}(a_2-\kappa a_1)$ and $\beta=(a_2-\kappa a_1)^{-1}\kappa^2$. This case has been treated recently in \cite{Golovaty2018}.   In particular, if $a_1=0$, then $\alpha=1$ and $\beta=\kappa^2a_2^{-1}$. So we obtain the new approximation to  the classic $\delta'$-interactions of strength $\beta$.

\smallskip

\paragraph{\textit{Exotic point interactions.}}
The case A1 also contains a few  types of specific limit point interactions. For instance,  if we choose potential $q$ such that $a_2=0$, but $a_1\neq 0$, then the limit operator $\mathcal{H}$ is associated with the point interactions
      \begin{equation*}
      e^{i \phi}   \begin{pmatrix}
           \beta  &    -\alpha\\
            \alpha^{-1}   &  0
           \end{pmatrix},
        \end{equation*}
where $\alpha=-|\kappa| |a_1|^{-1}$,
$\beta=|\kappa|^{-1}a_1^{-1}(|\kappa|^2a_0-2\Re(\overline{\kappa}a_1))$  and $\phi=\arg (a_2-\overline{\kappa} a_1)$.

In the case when $a_2=2\Re(\overline{\kappa}a_1)-|\kappa|^2a_0$, operators $H_\eps$ converge to the operator which describes the point interactions
      \begin{equation*}
     e^{i \phi}   \begin{pmatrix}
            0     &  -\alpha\\
            \alpha^{-1} &      \beta
        \end{pmatrix},
\end{equation*}
where $\alpha=-|\kappa|^2|a_2-\overline{\kappa} a_1|^{-1}$ and $\beta=a_2|a_2-\overline{\kappa} a_1|^{-1}$ and
$\phi=\arg ({a}_1\kappa \bar-a_0|\kappa|^2)$.

\smallskip

It is worth noting that Theorem~\ref{ThmS} provides the approximation to almost all point interactions given by  separated boundary conditions \eqref{SepConds} -- the Robin-Robin type (case B1), the Neumann-Neumann (case B1 with $q=0$), the Robin-Dirichlet and Dirichlet-Robin types (case B2) and the Dirichlet-Dirichlet type (case B3).

\section{Half-Bound States}
Let us consider the operator
$$
   B=-\frac{d^2}{dx^2}+\la g,\,\cdot\,\ra\,f(x)+\la f,\,\cdot\,\ra\,g(x),\quad \dom B=W_2^2(\Real),
$$
in  $L_2(\Real)$. We will introduce the notion of half-bound state, which plays a crucial role in our considerations.

\begin{defn}
We say that the  operator~$B$  possesses a half-bound state provided there exists a nontrivial solution of the equation  $-u''+\la g, u\ra f+\la f,u\ra g= 0$ that is bounded on the whole line.
\end{defn}
In this case we also say that $B$ has \textit{a zero-energy resonance.}
All half-bound states of $B$ form a linear space.

\begin{lem}\label{LemmaHBS}
 The operator $B$ possesses a half-bound state if and only if
  \begin{equation}\label{HBScondition}
   \|g_0\fpr{f}-f_0 \fpr{g}\|^2=2\Re{(f_0\bar{g}_0)}.
  \end{equation}
The operator can possess one or two linearly independent half-bound states.

\noindent
\textit{(i)}  If  \eqref{HBScondition} holds and $f_0g_0\neq 0$, then $B$  has the  half-bound state
\begin{equation}\label{HBS1}
  \sigma=|g_0|^2\left(\bar{f}_0\spr{f}-\la f,\spr{f}\ra\right)
  -|f_0|^2\left(\bar{g}_0\spr{g}-\la g,\spr{g}\ra\right).
\end{equation}

\noindent
\textit{(ii)} If $f_0=0$, $g_0=0$ and $\pi\neq 0$, then  the only constant function is a  half-bound state of $B$.

\noindent
\textit{(iii)}\;{\rm(Double zero-energy resonance)}
 If  $f_0=0$, $g_0=0$ and $\pi=0$, then
there exist two linearly independent half-bound states of $B$, namely the constant function and
\begin{equation*}
 \omega=e^{i\arg(\la \fpr{f},\fpr{g}\ra+1)}\|\fpr{g}\|\, \spr{f}-\|\fpr{f}\|\, \spr{g}.
\end{equation*}
\end{lem}

\begin{proof}
Equation $Bu=0$  has the general solution
  \begin{equation*}
  u=c_1\spr{f}+c_2\spr{g}+c_3+c_4 x,
  \end{equation*}
where the constants $c_k$  satisfy  conditions
\begin{align*}
  &\la f,\spr{f}\ra\,c_1+\bigl(\la f,\spr{g}\ra-1\bigr)\,c_2+\bar{f}_0c_3+\bar{f}_1c_4=0,  \\
  &\bigl(\la g,\spr{f}\ra-1\bigr)\,c_1+\la g,\spr{g}\ra\,c_2
  +\bar{g}_0c_3+\bar{g}_1c_4=0.
\end{align*}
These conditions are derived from the equation in view of the linear independence of  $f$ and $g$.
Next,  $u(x)=c_3+c_4 x$ for negative $x$ with  large absolute value,
since  $\spr{f}$ and $\spr{g}$ vanish in a neighbourhood of the negative infinity. Therefore $c_4=0$, because we  look for bounded solutions.
On the other hand,
\begin{equation}\label{fgLargeX}
  \spr{f}(x)=f_0x-f_1, \qquad  \spr{g}(x)=g_0x-g_1
\end{equation}
for large positive $x$.
Indeed, if  $x$ lies on the right  of $\supp f$, then
\begin{equation*}
  \spr{f}(x)=\int_{-\infty}^x(x-s)f(s)\,ds
  =x\int_{\Real}f(s)\,ds-\int_{_{\Real}}sf(s)\,ds=f_0x-f_1.
\end{equation*}
We conclude from \eqref{fgLargeX} that $u(x)=(c_1f_0+c_2g_0)x-c_1f_1-c_2g_1+c_3$
for large positive $x$, and hence that $f_0 c_1+g_0 c_2=0$, since $u$ is bounded.
Therefore  vector $\vec{c}=(c_1,c_2,c_3)$ must be a non-trivial solution of the homogeneous linear system $A\vec{c}=0$ with  matrix
\begin{equation*}
  A=
  \begin{pmatrix}
    \la f,\spr{f}\ra & \la f,\spr{g}\ra-1 &\bar{f}_0 \\
    \la g,\spr{f}\ra-1 & \la g,\spr{g}\ra & \bar{g}_0\\
    f_0 & g_0 & 0
  \end{pmatrix}.
\end{equation*}
For any zero mean functions $v$, $w$ of compact support, integrating by parts yields
\begin{equation}\label{ff2=f1f1}
\la v,\spr{w}\ra=-\la\fpr v, \fpr w\ra, \qquad \la v,\spr{v}\ra=-\|\fpr v\|^2.
\end{equation}
 Then a direct calculation verifies
\begin{multline*}
  \det A= -\la g_0f-f_0 g, g_0\spr{f}-f_0 \spr{g}\ra-2\Re{(f_0\bar{g}_0)}\\
  =\|g_0\fpr{f}-f_0 \fpr{g}\|^2-2\Re{(f_0\bar{g}_0)},
\end{multline*}
since $g_0f-f_0 g$ is a function of zero mean.
Hence operator $B$ possesses a half-bound state if and only if \eqref{HBScondition} holds, i.e., $\lambda=0$.

Suppose that $f_0g_0\neq 0$ and matrix $A$ is degenerate. It can only happen if the first and second rows of $A$ are linearly dependent. In particular, from this we deduce
\begin{equation}\label{f0g0relation}
\bar{g}_0\left(\la f,\spr{g}\ra-1\right) =\bar{f}_0\,\la g,\spr{g}\ra.
\end{equation}
Next, vector $(g_0,-f_0,c_3)$ satisfies the third equation of  system $A\vec{c}=0$ for all $c_3$. Substituting the vector to the first equation yields
\begin{equation*}
  g_0\la f,\spr{f}\ra-f_0(\la f,\spr{g}\ra-1)+\bar{f}_0c_3=0.
\end{equation*}
From \eqref{f0g0relation} we have
 $\bar{f}_0\bar{g}_0c_3=|f_0|^2\la g,\spr{g}\ra-|g_0|^2\la f,\spr{f}\ra$.
Hence the vector
\begin{equation*}
\left(|g_0|^2\bar{f}_0,\:-|f_0|^2\bar{g}_0,\:|f_0|^2\la g,\spr{g}\ra-|g_0|^2\la f,\spr{f}\ra\right)
\end{equation*}
solves   $A\vec{c}=0$ and therefore function
\begin{multline*}
 \sigma(x)=|g_0|^2\bar{f}_0 \,\spr f(x)-|f_0|^2\bar{g}_0\,\spr g(x) +|f_0|^2\la g,\spr{g}\ra-|g_0|^2\la f,\spr{f}\ra\\
 =
 |g_0|^2\left(\bar{f}_0\,\spr{f}(x)-\la f,\spr{f}\ra\right)
  -|f_0|^2\left(\bar{g}_0\,\spr{g}(x)-\la g,\spr{g}\ra\right).
\end{multline*}
is a half-bound state of $B$.

In the cases \textit{(ii)} and \textit{(iii)}, $f$ and $g$ have zero means. The matrix $A$ becomes
\begin{equation}\label{MatrixAinNfNg}
  A=-
  \begin{pmatrix}
      \|\fpr{f}\|^2& \la \fpr{f},\fpr{g}\ra +1 & 0 \\
    \la\fpr{g},\fpr{f}\ra +1 & \|\fpr{g}\|^2 &   0\\
     0 & 0 & 0
  \end{pmatrix},
\end{equation}
by \eqref{ff2=f1f1}. The rank of $A$ equals $1$ if  and only if
$\|\fpr{f}\|\,\|\fpr{g}\| =|\la\fpr{f},\fpr{g}\ra+1|$, i.e., $\pi=0$, where $\pi$ is given by \eqref{Pi}.
Then the kernel of $A$ is the span of two vectors $\vec{c}_1=(0,0,1)$ and $\vec{c}_2=(e^{i\vartheta}\|\fpr{g}\|, -\|\fpr{f}\|,0)$, where $\vartheta=\arg (\la \fpr{f},\fpr{g}\ra+1)$. In fact, substituting $\vec{c}_2$ into the first equation gives us
\begin{multline*}
 e^{i\vartheta} \|\fpr{f}\|^2\|\fpr{g}\|-\|\fpr{f}\|(\la \fpr{f},\fpr{g}\ra +1)\\
 =\|\fpr{f}\|\left(e^{i\vartheta}\,|\la \fpr{f},\fpr{g}\ra +1|-
 \la \fpr{f},\fpr{g}\ra -1\right)=0.
\end{multline*}
Hence operator $B$ possesses half-bound states $1$ and $\omega$.
Of course, the only constant function is a half-bound state of $B$ if $\pi\neq 0$.
\end{proof}

The last lemma partially explains the origination of some conditions in Theorems~\ref{ThmC} and \ref{ThmS}.

\section{Auxiliary statements}

From now on,  we  assume that the supports of $f$, $g$ and $q$ lie in  interval $\cI=[-1,1]$. This involves no loss of generality.
Then a half-bound state of  $B$ is  constant  outside the interval $\cI$ as a solution of equation $u''=0$, which is bounded at infinity (see Fig.~\ref{FigHBS}). Therefore the restriction of $u$ to $\cI$ is a nonzero solution of the Neumann boundary value problem
\begin{equation}\label{NeumanProblem}
     -u''+(g, u)\, f+(f,u)\, g= 0\quad \text{in } \cI, \quad   u'(-1)=0, \; u'(1)=0,
\end{equation}
where  $(\cdot,\cdot)$ is the scalar product in $L_2(\cI)$.

\begin{figure}[h]
  \centering
  \includegraphics[scale=0.85]{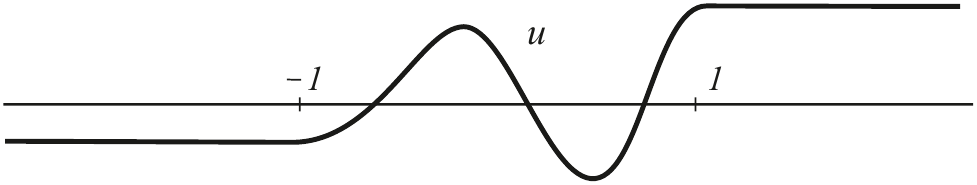}\\
  \caption{Half-bound state of $B$.}\label{FigHBS}
\end{figure}

Given $r\in L_2(\cI)$ and $a, b\in \mathbb{C}$, we consider the nonhomogeneous problem
  \begin{equation}\label{NHProblem}
  -v ''+(g, v)\, f+(f, v)\, g=r\quad \text{in } \cI, \quad v'(-1)=a, \; v'(1)=b.
\end{equation}
If operator $B$ has a half-bound state, i.e., \eqref{NeumanProblem} admits a non-trivial solution, then  problem \eqref{NHProblem} is generally unsolvable. In this case, even if the problem has a solution for some $r$, $a$ and $b$, this solution is ambiguously determined, according to Fredholm's alternative. But we can always choose a solution of  \eqref{NHProblem} for which  the estimate
\begin{equation}\label{NHPUestimate}
  \|v\|_{W_2^2(\cI)}\leq c(|a|+|b|+\|r\|_{L_2(\cI)})
\end{equation}
holds with some constant $c$ depending only on $f$ and $g$.

\begin{prop}\label{PropSolvabilityOfNHP}
(i) Suppose that operator $B$ possesses the half-bound state $\sigma$ given by \eqref{HBS1}. Then  problem \eqref{NHProblem} admits a solution if and only if
\begin{equation}\label{SolvabilityOfNHPSigma}
  a\bar{\sigma}_--b\bar{\sigma}_+=(\sigma,r).
\end{equation}

(ii)
If the only constant function is a  half-bound state of $B$, then problem \eqref{NHProblem} is solvable iff $a-b=(1,r)$.

(iii) If $B$ has the double zero-energy resonance, then  \eqref{NHProblem}  is solvable iff
\begin{equation}\label{SolvabilityOfNHP}
a-b=(1,r), \quad b\overline{\kappa}=-(\omega, r).
\end{equation}

(iv) Suppose that  \eqref{NHProblem} is solvable for given $a$, $b$ and $r$. Then it always admits a solution that satisfies \eqref{NHPUestimate}.

\end{prop}

\begin{proof}
Parts \textit{(i)--(iii)} are a simple consequence of Fredholm's alternative for
the self-adjoint operator
\begin{equation*}
B_0=-\frac{d^2}{dx^2}+(g,\,\cdot\,) \,f+(f,\,\cdot\,) \,g,
\quad \dom B_0=\big\{h\in W_2^2(\cI)\colon h'(-1)=0, \; h'(1)=0\big\}
\end{equation*}
in  space $L_2(\cI)$. For instance, two solvability conditions \eqref{SolvabilityOfNHP}
can be easy obtained by multiplying equation \eqref{NHProblem} by $1$ and $\overline{\omega}$ in turn and then integrating by parts twice  in view of the boundary conditions. According to
Fredholm's alternative, these conditions are also sufficient.

The problem \eqref{NeumanProblem} has a trivial solution only, if
$\lambda\neq0$. Then \eqref{NHProblem} is uniquely solvable for all $a$, $b$, $r$, and the solution satisfies \eqref{NHPUestimate}.
Otherwise, there are infinitely many solutions of \eqref{NeumanProblem}, and  therefore \eqref{NHProblem} is solvable under the conditions stated above. The proof of part \textit{(iv)} is  similar for all the cases \textit{(i)--(iii)} of non-uniqueness. We will focus our attention on more difficult case \textit{(iii)}. Now since $f_0=0$ and $g_0=0$, antiderivatives  $\fpr{f}$ and $\fpr{g}$ have compact supports lying in $\cI$. Hence
\begin{equation}\label{fgAtpm1}
  \fpr{f}(-1)=0, \quad \fpr{g}(-1)=0, \quad
\fpr{f}(1)=0,\quad\fpr{g}(1)=0.
\end{equation}

Let us find a partial solution of \eqref{NHProblem} of the form
\begin{equation*}
   v_*=c_1\spr f+c_2\spr g-\spr r+a(x+1),
\end{equation*}
where $r^{(-2)}(x)=\int_{-1}^x(x-s)r(s)\,ds$. For all $c_1$ and $c_2$ function $v_*$ satisfies boundary conditions in \eqref{NHProblem}. Indeed, from \eqref{fgAtpm1} and the first solvability condition in \eqref{SolvabilityOfNHP} we see that
\begin{align*}
  &v_*'(-1)=c_1\fpr{f}(-1)+c_2\fpr{g}(-1)-\fpr r(-1)+a=a,\\
  &v_*'(1)=c_1\fpr{f}(1)+c_2\fpr{g}(1)-\fpr r(1)+a=a-(1,r)=b,
\end{align*}
since $\fpr r(1)=(1,r)$.
Let us introduce the temporary notation $n_f$, $n_g$ and $p$ for $\|\fpr{f}\|$, $\|\fpr{g}\|$ and $\la\fpr{f}, \fpr{g}\ra$ respectively.
Next, from \eqref{fgLargeX}, which now holds for $x\geq 1$, we have
$\spr f(1)=-f_1$, $\spr g(1)=-g_1$. Then
\begin{equation}\label{Kappa}
  \kappa=\omega(1)=e^{i\vartheta}n_g \spr f(1)-n_f\spr g(1)=n_fg_1 -e^{i\vartheta}n_gf_1,
\end{equation}
where $\vartheta=\arg (p+1)$. Also
\begin{equation}\label{omegaR}
  (\omega'',\spr r)=-\overline{\kappa} \fpr r(1)+(\omega, r)=-\overline{\kappa} (1,r)+(\omega, r)=-(\kappa-\omega,r).
\end{equation}

Direct substitution $v_*$ into equation \eqref{NHProblem} yields the linear system
\begin{equation}\label{LinSysK}
  \begin{pmatrix}
      n_f^2  & p+1  \\
      \bar{p}+1 & n_g^2
  \end{pmatrix}
  \begin{pmatrix}
    c_1\\c_2
  \end{pmatrix}
  =
  \begin{pmatrix}
  z_1 \\
  z_2
  \end{pmatrix}
\end{equation}
with $z_1=a \bar{f}_1-(f, \spr r)$ and $z_2=a \bar{g}_1-(g, \spr r)$
(cf. \eqref{MatrixAinNfNg} in the proof of Lemma~\ref{LemmaHBS}).
Since $\pi=0$ in this case, we have $|p+1|= n_fn_g$, and
the system  is consistent, if and only if $(p+1)z_2 =n_g^2 z_1$. This condition can be written in the form $e^{i\vartheta}n_fz_2= n_gz_1$.
Recalling \eqref{Kappa} and \eqref{omegaR} gives us
\begin{multline*}
e^{i\vartheta}n_fz_2 -n_gz_1=e^{i\vartheta}n_f(a \bar{g}_1-(g, \spr r))-n_g(a \bar{f}_1-(f, \spr r))
 \\ =e^{i\vartheta}a(n_f\bar{g}_1 -e^{-i\vartheta}n_g\bar{f}_1)
 +e^{i\vartheta}(e^{i\vartheta}n_gf-n_fg, \spr r)\\=e^{i\vartheta}\left(a \overline{\kappa} +(\omega'',\spr r)\right)
 = e^{i\vartheta} \big(a\overline{\kappa}-(\kappa-\omega,r)\big)=0,
\end{multline*}
because  $a\overline{\kappa}=(\kappa,r)+b\overline{\kappa}=(\kappa-\omega, r)$ by  solvability conditions  \eqref{SolvabilityOfNHP}.
Hence system \eqref{LinSysK} is solvable and admits solution
$\vec{c}=(n_f^{-2}z_1,0)$. Therefore
\begin{equation}\label{ExplisitSolV}
  v_*=n_f^{-2}\left(a \bar{f}_1-(f, \spr r)\right)\spr{f}-\spr{r}+a(x+1).
\end{equation}
is a solution of \eqref{NHProblem}.
In the case $\kappa\neq 0$, we can modify  $v_*$ to obtain the solution
\begin{equation}\label{ExplisitSolKneq0}
  v=v_*+\kappa^{-1}\left(f_1n_f^{-2}\big(a \bar{f}_1-(f, \spr r)\big)+\spr{r}(1)-2a\right) \omega
\end{equation}
of   \eqref{NHProblem} satisfying $v(-1)=0$ and $v(1)=0$, as is easy to check. If $\kappa=0$, then we write
\begin{equation}\label{ExplisitSolK0}
v=v_*-(\omega,\omega)^{-1}(\omega,v_* )\,\omega;
\end{equation}
this solution fulfills the following conditions $v(-1)=0$ and
$(\omega,v)=0$.

Estimate \eqref{NHPUestimate} for $v$ immediately follows from   explicit formulae \eqref{ExplisitSolV}, \eqref{ExplisitSolKneq0}  and the continuity of  operator $L_2(\cI)\ni r\mapsto \spr{r}\in W_2^2(\cI)$. Indeed, all terms in \eqref{ExplisitSolV} and \eqref{ExplisitSolKneq0}  contain either $a$ or $\spr r$. For instance, we can estimate
\begin{gather*}
  \big\|a \bar{f}_1 n_f^{-2}\spr{f}\big\|_{W_2^2(\cI)}\leq c_1|a|\big\|\spr{f}\big\|_{W_2^2(\cI)}\leq c_2|a|\|f\|_{L_2(\cI)}\leq c_3|a|, \\
  \begin{aligned}
  \big\|n_f^{-2}(f, \spr r)\spr{f}\big\|_{W_2^2(\cI)}&\leq
  c_4\big\|\spr{f}\big\|_{W_2^2(\cI)}|(f, \spr r)|\\
  &\leq
  c_5 \|f\|_{L_2(\cI)}\big\|\spr{r}\big\|_{L_2(\cI)}
  \leq c_6\|r\|_{L_2(\cI)},
  \end{aligned}
\end{gather*}
etc.
\end{proof}

\begin{rem}\label{RemarkOnSolvability}
In case A1, $\kappa\neq 0$, a slight variant of the proof above provides the estimate
\begin{equation}\label{NHPVestimateOnlyR}
  \|v\|_{W_2^2(\cI)}\leq c\|r\|_{L_2(\cI)}
\end{equation}
for the solution given by \eqref{ExplisitSolKneq0}.  In fact, owing to
\eqref{SolvabilityOfNHP} we see that for any $r\in L_2(\cI)$ there exists a unique pair of numbers
\begin{equation}\label{ABasR}
 a(r)=(1-\kappa^{-1}\omega,r), \qquad b(r)=-\overline{\kappa}^{-1}(\omega, r)
\end{equation}
for which \eqref{NHProblem} is solvable. These numbers
can be regarded as  linear functionals in $L_2(\cI)$. Hence  we have the bounds $|a(r)|+|b(r)|\leq c \|r\|_{L_2(\cI)}$, from which \eqref{NHPVestimateOnlyR} follows.

On the other hand, if $\kappa=0$ in \eqref{SolvabilityOfNHP}, then \eqref{NHProblem} is solvable only for $r$ such that $(\omega, r)=0$.
It is interesting to  note that even  single equation \eqref{NHProblem}, without any boundary conditions, is  unsolvable if this condition does not hold.
\end{rem}

\begin{prop}\label{PropSigmaPrps}
  The half-bound state $\sigma$ given by \eqref{HBS1} possesses the following properties:
\begin{itemize}
  \item[\textit{(i)}] the limit $\sigma_+=\lim_{x\to+\infty}\sigma(x)$ is always a real number;
  \item[\textit{(ii)}] $\sigma_+-\sigma_-=\bar{f}_0\bar{g}_0\,(f_0g_1-f_1g_0)$.
\end{itemize}
\end{prop}
\begin{proof}
As above, we assume that $f$ and $g$ vanish outside of $\cI$.
Then $\fpr f(1)=f_0$.
We also note  that $\spr f(1)=f_0-f_1$,
by~\eqref{fgLargeX}, and also that
\begin{equation*}
(f,\spr{f})=\bar{f}_0\,(f_0-f_1)-\|\fpr{f}\|^2_{L_2(\cI)},
\end{equation*}
where $\|\,\cdot\,\|_{\cI}$ is a norm in $L_2(\cI)$. In fact,
\begin{equation*}
\int_{-1}^1\bar{f}\spr{f}\,dx
  =\fpr{\bar{f}}(1)\spr{f}(1)-\int_{-1}^1|\fpr{f}|^2\,dx
  =
  \bar{f}_0(f_0-f_1)-\|\fpr{f}\|^2_{\cI}.
\end{equation*}
The same formulae are valid for $g$  as well.
We have
\begin{multline*}
  \sigma_+=\sigma(1)=|g_0|^2\left(\bar{f}_0(f_0-f_1)-(f,\spr{f})\right)
  \\
  -|f_0|^2\left(\bar{g}_0(g_0-g_1)-(g,\spr{g})\right)
  =|g_0|^2\,\|\fpr{f}\|^2_{\cI}-|f_0|^2\,\|\fpr{g}\|^2_{\cI}.
\end{multline*}
and therefore $\sigma_+\in\Real$. Next write
\begin{multline*}
  \sigma_-=\sigma(-1)=-|g_0|^2\,(f,\spr{f})+|f_0|^2\,(g,\spr{g}) \\
  =-|g_0|^2\,\left(\bar{f}_0\,(f_0-f_1)-\|\fpr{f}\|^2_{L_2(\cI)}\right)
  +|f_0|^2\,\left(\bar{g}_0\,(g_0-g_1)-\|\fpr{g}\|^2_{L_2(\cI)}\right)\\
  =|g_0|^2\,\|\fpr{f}\|^2_{\cI}-|f_0|^2\,\|\fpr{g}\|^2_{\cI}+
  |g_0|^2\bar{f}_0f_1-|f_0|^2\bar{g}_0g_1\\=
  \sigma_+-\bar{f}_0\bar{g}_0\,(f_0g_1-f_1g_0),
\end{multline*}
which establishes \textit{(ii)}.
\end{proof}

At the end of the section,  we record some technical assertion.
Let $[w]_{\xi}$ denote  the jump $w(\xi+0)-w(\xi-0)$ of  function $w$ at a point $\xi$.

\begin{prop}\label{PropW22Corrector}
Let $U$ be the real line with two removed points $x=-\eps$ and $x=\eps$, i.e., $U=\Real\setminus \{-\eps,\eps\}$.
Assume that function $w\in W_{2, loc}^2(U)$ along with its first derivative has jump discontinuities at points $x=-\eps$ and $x=\eps$. There exists a function $\rho\in C^\infty(U)$ such that   $w+\rho$ belongs to $W_{2, loc}^2(\Real)$. Moreover, $\rho$ is a function of compact support, $\rho$ vanishes in $(-\eps,\eps)$  and
    \begin{equation}\label{REst}
        |\rho^{(k)}(x)|\leq C \Bigl(\bigl|[w]_{-\eps}\bigr|+\bigl|[w]_{\eps}\bigr|
        +\bigl|[w']_{-\eps}\bigr|+\bigl|[w']_{\eps}\bigr|\Bigr)
    \end{equation}
for $|x|\geq \eps$,  $k=0,1,2$, where the constant $C$ does not depend on $w$ and $\eps$.
\end{prop}
\begin{proof}
Let us introduce functions $\phi$ and $\psi$ that are smooth outside the origin, have compact supports contained in $[0,\infty)$, and such that $\phi(+0)=1$, $\phi'(+0)=0$, $\psi(+0)=0$ and $\psi'(+0)=1$.
We set
\begin{equation*}
\rho(x)=[w]_{-\eps}\, \phi(-x-\eps)-[w']_{-\eps}\,\psi(-x-\eps)\\
-[w]_{\eps}\,\phi(x-\eps)-[w']_{\eps}\,\psi(x-\eps).
\end{equation*}
By construction,  $\rho$ has a compact support and vanishes in $(-\eps,\eps)$. An easy computation also shows that
\begin{equation*}
  [\rho]_{-\eps}=-[w]_{-\eps}, \quad[\rho]_\eps=-[w]_\eps, \quad[\rho']_{-\eps}=-[w']_{-\eps}, \quad [\rho']_\eps=-[w']_\eps.
\end{equation*}
Therefore $w+\rho$ is continuous on~$\Real$ along with the first derivative and consequently belongs to $W_{2, loc}^2(\Real)$.
Finally, the explicit formula for $\rho$  makes it obvious that inequality \eqref{REst} holds.
\end{proof}

\section{Proof of Theorem \ref{ThmC}}

\subsection{How to guess the limit operator}\label{Subsect41}
Given $h\in L_2(\Real)$ and $\zeta\in \mathbb{C}$ with  $\Im\zeta\neq0$, we set $y_\eps=(H_\eps-\zeta)^{-1}h$.
Let us find a formal asymptotics of $y_\eps$, as $\eps\to 0$, in the form
\begin{equation}\label{AsymptoticsYepsC1}
 y_\eps(x)\sim
  \begin{cases}
      y(x)+\cdots   & \text{if }|x|>\eps,\\
      u\xe+\eps v\xe+\cdots & \text{if }  |x|<\eps,
 \end{cases}
\end{equation}
provided the coupling conditions $[y_\eps]_{\pm\eps}=0$, $[y'_\eps]_{\pm\eps}=0$ hold.
Function $y_\eps$ is a $L_2(\Real)$-solution of the equation
\begin{equation*}
 -y_\eps''+V(x)y_\eps+\eps^{-3}Q_\eps y_\eps+\eps^{-1}q(\eps^{-1} x)y_\eps=\zeta y_\eps+h\quad  \text{in }  \Real.
\end{equation*}
Since the interval  on which the perturbation is localized shrinks to a point,    $y$ must solve the equation
\begin{equation}\label{C1AsymptoticsEq}
-y''+V(x)y=\zeta y+h \qquad\text{in } \Real\setminus\{0\}
\end{equation}
and, of course, it must belong to $L_2(\Real)$. This solution can not be uniquely determined without additional conditions at the origin. One naturally expects that these conditions depend on the perturbation.

Set $t=\eps^{-1}x$ and $z_\eps(t)=y_\eps(\eps t)$.
Then, for $|t|<1$, we have
\begin{equation*}
  -\frac{d^2z_\eps}{dt^2}+(g, z_\eps)\, f(t)+(f,z_\eps) \,g(t) +\eps q(t)z_\eps=\eps^2 \big(\zeta z_\eps+V(\eps t)+h(\eps t) \big)
\end{equation*}
Since $z_\eps\sim u+\eps v+\cdots$, we see that
$-u''+Qu=0$ and $-v''+Qv=-qu$ for $t\in\cI$, where $Q=(g,\,\cdot\,) \,f+(f,\,\cdot\,)\,g$ is a rank-two operator in $L_2(\cI)$.
Next, the asymptotic equalities $y(\pm\eps)\sim u(\pm1)+\eps v(\pm 1)+\cdots$ and $y'(\pm\eps)\sim \eps^{-1}u'(\pm1)+v'(\pm 1)+\cdots$
imply in particular that
\begin{equation}\label{CouplingYV0}
y_-=u(-1),\qquad y_+=u(1),
\end{equation}
and also that  $u'(-1)=0$, $u'(1)=0$, $v'(-1)=y'_-$ and $v'(1)=y'_+$.
Here and subsequently, $y_\pm=y(\pm 0)$ and $y'_\pm=y'(\pm 0)$.

Com\-bi\-ning the equalities above, we obtain two boundary value problems
\begin{align}\label{bvpV0}
  &-u''+Qu=0,\quad t\in\cI,  && u'(-1)=0, \quad u'(1)=0;
  \\\label{bvpV1}
  &-v'' +Qv=-q u,\quad t\in\cI,  && v'(-1)=y'_-, \quad v'(1)=y'_+.
\end{align}

\smallskip
\paragraph{\textit{Case A1}}
In view of Lemma~\ref{LemmaHBS} \textit{(iii)} problem \eqref{bvpV0} has the two-dimensional space of solutions generated by $1$ and $\omega$.
We set
\begin{equation}\label{VInCaseKappaNotZero}
  u(t)=y_-+\kappa^{-1}\big(y_+-y_-\big)\,\omega(t),\quad t\in\cI,
\end{equation}
provided $\kappa\neq 0$. Recall that $\omega(1)=\kappa$. Hence, $u$ is a restriction of  half-bound state to $\cI$ such that $u(-1)=y_-$ and $u(1)=y_+$. Problem \eqref{bvpV1} with the introduced $u$ in the right hand side of the equation is solvable if conditions \eqref{SolvabilityOfNHP} hold, namely
$y'_--y'_+=-(1,q u)$ and
$\overline{\kappa} \,y'_+ =(\omega,q u)$.
We now  substitute \eqref{VInCaseKappaNotZero} into the last equalities and recall notation \eqref{Ak}. After some calculations we thus  write the solvability conditions  in the matrix form
\begin{equation}\label{MatrixEquality}
  \begin{pmatrix}
    \frac{a_1}{\kappa} & -1\\
    \frac{a_2}{|\kappa|^2} & -1\\
  \end{pmatrix}
  \begin{pmatrix}
    y_+\\ y'_+
  \end{pmatrix}=
  \begin{pmatrix}
    \frac{a_1-\kappa a_0}{\kappa} & -1\\
    \frac{a_2-\kappa \bar{a}_1}{|\kappa|^2} & 0\\
  \end{pmatrix}
  \begin{pmatrix}
    y_-\\ y'_-
  \end{pmatrix}.
\end{equation}
Since $a_2\neq\overline{\kappa} a_1$ in  case A1, the matrix on the left is invertible. From this we  deduce
\begin{equation}\label{C1CouplingCnds}
  \begin{pmatrix}
    y_+\\ y'_+
  \end{pmatrix}=\frac{1}{a_2-\overline{\kappa} a_1}
    \begin{pmatrix}
     |\kappa|^2a_0-2{\rm Re}(\overline{\kappa} a_1)+a_2
     &   |\kappa|^2\\ a_0a_2-|a_1|^2 &    a_2
     \end{pmatrix}
  \begin{pmatrix}
    y_-\\ y'_-
  \end{pmatrix}.
\end{equation}
As $a_2-\overline{\kappa} a_1=e^{i \arg (a_2-\overline{\kappa} a_1)}|a_2-\overline{\kappa} a_1|$ and $e^{-i \arg (a_2-\overline{\kappa} a_1)}=e^{i \arg (a_2-\kappa \bar{a}_1)}$, we see that function $y$ in asymptotics \eqref{AsymptoticsYepsC1} must be a solution of \eqref{C1AsymptoticsEq} belonging to $\mathcal{V}$ and satisfying conditions \eqref{MatrixHinC1}. Since coupling conditions \eqref{C1CouplingCnds} are simultaneously the sol\-va\-bility conditions for \eqref{bvpV1}, there exists a solution $v$ of this problem defined up to terms $c_1+c_2\omega$. We can fix $v$ such that $v(-1)=0$ and $v(1)=0$ (see \eqref{ExplisitSolKneq0}).

As we see in Fig.~\ref{FigBranchingAlgorithm}, there is  also another path going to node A1, which is described by  conditions   $\kappa=0$ and $a_2\neq0$. Since $\kappa=0$, half-bound state $\omega$ now vanishes not only at $t=-1$, but also at $t=1$.  Hence $\omega$ is a bound state  of operator $B$. Then for any solution $u=c_1+c_2 \omega$ of \eqref{bvpV0} we have $u(-1)=u(1)=c_1$. Therefore
\begin{equation}\label{C1Kappa=0Cnd0}
   y_+=y_-
\end{equation}
by \eqref{CouplingYV0},  and $u=y(0)+c_2\omega$. As above, applying  solvability conditions \eqref{SolvabilityOfNHP} to problem \eqref{bvpV1} yields
\begin{equation}\label{SolvCondsKappaZero}
y'_+-y'_-=a_0y(0)+a_1c_2,\qquad \bar{a}_1y(0)+a_2c_2=0,
\end{equation}
 from which we have
\begin{equation}\label{C1Kappa=0Cnd1}
  y'_+=y'_-+a_2^{-1}(a_0a_2-|a_1|^2)\:y(0).
\end{equation}
Hence, in the case $\kappa=0$, the leading term $y$ in asymptotics for $y_\eps$ is a solution of \eqref{C1AsymptoticsEq} obeying conditions \eqref{C1Kappa=0Cnd0} and \eqref{C1Kappa=0Cnd1} (cf. coupling conditions \eqref{MatrixHinC1} for
$\kappa=0$). We  also have $u=y(0)\left(1-\bar{a}_1a_2^{-1}\,\omega\right)$.
Our choice of $y$ ensures the solvability of \eqref{bvpV1}; we also fix $v$ as in \eqref{ExplisitSolK0}.

\smallskip
\paragraph{\textit{Case A2}} Now problem \eqref{bvpV0} admits a one-parametric family of solutions $u=c_0\sigma$, where $\sigma$ is given by \eqref{HBS1}. Applying \eqref{CouplingYV0} and solvability condition \eqref{SolvabilityOfNHPSigma} for problem \eqref{bvpV1}, we deduce
$y_-=c_0\sigma_-$, $ y_+=c_0\sigma_+$ and $\bar{\sigma}_+y'_+-\bar{\sigma}_-y'_-=c_0\sigma_*$, because the limits
$\sigma_-$ and $\sigma_+$ in \eqref{sigmapm} coincide with values $\sigma(-1)$ and $\sigma(1)$ respectively. Since both the limits $\sigma_-$ and $\sigma_+$ are different from zero, we have
\begin{equation*}
c_0=\frac{y_+}{\sigma_+}=\frac{y_-}{\sigma_-}, \qquad
  \bar{\sigma}_+y'_+-\bar{\sigma}_-y'_-=\frac{\sigma_*}{\sigma_-}\,y_-.
\end{equation*}
From this we readily deduce the conditions
\begin{equation*}
 y_+=\frac{\sigma_+}{\sigma_-}\,y_-, \qquad y'_+=\frac{\bar{\sigma}_-}{\sigma_+}\,y'_-
 +\frac{\sigma_*}{\sigma_+\sigma_-}\,y_-,
\end{equation*}
since $\sigma_+$ is real by Proposition~\ref{PropSigmaPrps} \textit{(i)}.
The last equalities can be written in matrix form \eqref{MatrixHinA2} using identities $\sigma_-=e^{i\arg \sigma_-}|\sigma_-|$ and $\bar{\sigma}_-=e^{-i\arg \sigma_-}|\sigma_-|$. Then problem \eqref{bvpV1} admits a solution $v$ such that $(\sigma, v)=0$.

\smallskip
\paragraph{\textit{Case A3}}
There two paths going to node A3 in the graph (see Fig.~\ref{FigBranchingAlgorithm}).  If $\pi=0$, $\kappa=0$, $a_1=0$ and $a_2=0$, then \eqref{C1Kappa=0Cnd0},
\eqref{SolvCondsKappaZero} reduce to the coupling conditions
\begin{equation}\label{CondsA2}
  y_+=y_-, \qquad y'_+=y'_-+a_0y(0)
\end{equation}
that correspond to  point interactions \eqref{MatrixHinC2}. In this case,  $u=y(0)$ and $v$ solves \eqref{bvpV1} and satisfies additional conditions $v(-1)=0$ and $(v,\omega)=0$.

Going the other way, for which $\pi\neq 0$,
we have that $u$ is a constant function, in view of Lemma~\ref{LemmaHBS} \textit{(ii)}.
Then \eqref{CouplingYV0} imply $y_+=y_-$ and $u=y(0)$. Next,  $v$ must be a solution of the problem
\begin{equation*}
  -v'' +Qv=-y(0)q,\quad t\in\cI, \quad  v'(-1)=y'_-, \quad v'(1)=y'_+,
\end{equation*}
which is solvable iff the second condition in \eqref{CondsA2} holds.
Hence, in this case  $y$ also satisfies the conditions \eqref{CondsA2}.
We fix a solution of \eqref{bvpV1} by condition $v(-1)=0$.

\subsection{Uniform approximation}

The basic idea of the proof is to construct a good approximation to  $y_\eps=(H_\eps-\zeta)^{-1}h$, uniformly for $h$ in bounded subsets of $L_2(\Real)$.
In addition, this approximation must belong to the domain of $H_\eps$.
The function $y=(\mathcal{H}-\zeta)^{-1}h$ is a  satisfactory approximation to $y_\eps$ for $|x|>\eps$, whereas the problem of finding a close approximation  on the support $(-\eps,\eps)$ of  perturbation is rather subtle. Recall that $\mathcal{H}$ stands for the limit operator in the case under study.

\smallskip
\paragraph{\textit{Case A1}}
Let $m$ be  a $L_2(\cI)$-function of zero mean such that
$(\omega,m)=1$. We consider the problem involving three parameters $\alpha_\eps$, $\beta_\eps$ and $\gamma_\eps$
\begin{equation}\label{C2ProblemVe1}
  -\vartheta_\eps ''+Q\, \vartheta_\eps=h(\eps \,\cdot)+\gamma_\eps m \quad \text{in } \cI,\qquad
   \vartheta'_\eps(-1)=\alpha_\eps,\quad \vartheta'_\eps(1)=\beta_\eps.
\end{equation}
The problem is solvable if and only if
\begin{equation*}
\kappa\alpha_\eps+\gamma_\eps=(\kappa-\omega,h(\eps \,\cdot)),
\qquad
\alpha_\eps-\beta_\eps=(1, h(\eps \,\cdot)),
\end{equation*}
provided operator $B$ has the double zero-energy resonance.

Assume  that  $\kappa\neq 0$. Let us introduce  function $\vartheta_\eps$ as a solutions of \eqref{C2ProblemVe1} with $\gamma_\eps=0$.
Then $\alpha_\eps$ and $\beta_\eps$ can be uniquely defined
$\alpha_\eps(h)=(1-\kappa^{-1}\omega,h(\eps \,\cdot))$, $\beta_\eps(h)=-\overline{\kappa}^{-1}(\omega, h(\eps \,\cdot))$
for given $h$ (see \eqref{ABasR}).
By \eqref{ExplisitSolKneq0}, there exists a unique solution $\vartheta_\eps$ satisfying the additional conditions $\vartheta_\eps(-1)=0$ and $\vartheta_\eps(1)=0$.

Set $w_\eps(x)=y(x)$ if $|x|>\eps$ and $w_\eps(x)=u\xe+\eps v\xe+\eps^2 \vartheta_\eps\xe$ if $|x|<\eps$.
By construction, $w_\eps$ belongs to $W_{2,loc}^2(\Real\setminus\{-\eps,\eps\})$, but this function is in general discontinuous at the points $x=\pm\eps$; its jumps and the jumps of its first derivative are small enough  as we will show below. This observation allows us to correct $w_\eps$  to a $W_{2,loc}^2(\Real)$-function by a small perturbation.
 We can find $\rho_\eps$ such that
\begin{equation}\label{YepsA1KappaNot0}
  Y_\eps(x)=w_\eps(x)+\rho_\eps(x)
 =\begin{cases}
      y(x)+\rho_\eps(x)   & \text{if }|x|>\eps,\\
      u\xe+\eps v\xe+\eps^2 \vartheta_\eps\xe& \text{if }  |x|<\eps
  \end{cases}
\end{equation}
belongs to $W_{2,loc}^2(\Real)$, by Proposition~\ref{PropW22Corrector}.  Recall that $\rho_\eps$ is zero in $(-\eps,\eps)$. Obviously $Y_\eps$ also belongs to the domain of $H_\eps$, since $y\in \mathcal{V}$ and $\rho_\eps$ has a compact support.

Now we suppose that  $\kappa=0$.
According to the second part of Remark~\ref{RemarkOnSolvability}, the solvability of \eqref{C2ProblemVe1} can not be ensured by parameters $\alpha_\eps$ and $\beta_\eps$ only. We can find  $\vartheta_\eps$  by setting $\alpha_\eps(h)=0$, $\beta_\eps(h)=-(1, h(\eps \,\cdot))$ and
$\gamma_\eps=-(\omega,h(\eps \,\cdot))$.
Then the problem  admits a unique solution such that $\vartheta_\eps(-1)=0$ and $(\omega, \vartheta_\eps)=0$, by \eqref{ExplisitSolK0}. Finally we define $Y_\eps\in \dom H_\eps$ by \eqref{YepsA1KappaNot0},  as for $\kappa\neq 0$.

\smallskip

\paragraph{\textit{Case A3}}
Approximation \eqref{YepsA1KappaNot0} constructed above for the case A1 with $\kappa=0$ is also suitable when $a_1=a_2=0$.
In the case when $\pi\neq0$, the double zero-energy resonance for $B$ is absent.  In view of Lemma~\ref{LemmaHBS} \textit{(ii)}, the only constant functions are half-bound states of $B$. Hence  \eqref{C2ProblemVe1} admits a solution if $ \alpha_\eps-\beta_\eps=(1,h(\eps \,\cdot))$
by Proposition~\ref{PropSolvabilityOfNHP} \textit{(ii)}.
We set $\alpha_\eps(h)=0$, $\gamma_\eps(h)=0$ and $\beta_\eps(h)=-(1, h(\eps \,\cdot))$, and fix the solution $\vartheta_\eps$ by additional condition $\vartheta_\eps(-1)=0$.

\smallskip
\paragraph{\textit{Case A2}} In this case operator $B$ possesses the half-bound state $\sigma$. Problem \eqref{C2ProblemVe1} admits a solution if and only if $\alpha_\eps(h)\bar{\sigma}_--\beta_\eps(h)\sigma_+=(\sigma, h(\eps \,\cdot))+\gamma_\eps(h) (\sigma, m)$. Set $\alpha_\eps(h)=0$, $\gamma_\eps(h)=0$ and $\beta_\eps(h)=-\sigma_+^{-1}(\sigma, h(\eps \,\cdot))$, for instance. Then we choose a solution $\vartheta_\eps$ in \eqref{YepsA1KappaNot0} such that $\vartheta_\eps(-1)=0$.

\smallskip
Regardless of the case under consideration, the values  $\alpha_\eps$, $\beta_\eps$ and $\gamma_\eps$  can be estimated by the norm of $h$:
\begin{equation}\label{EstAlphaBetaGamma}
  |\alpha_\eps(h)|+|\beta_\eps(h)|+|\gamma_\eps(h)|\leq c_{1}\|h(\eps\cdot)\|_{L_2(\cI)}\leq c_{2}\eps^{-1/2} \|h\|.
\end{equation}
Here we used the obvious estimate
\begin{equation*}
\int_{-1}^1|h(\eps s)|^2\,ds=\eps^{-1}\int_{-\eps}^\eps|h(x)|^2\,dx\leq \eps^{-1}\int_{\Real}|h(x)|^2\,dx.
\end{equation*}

\subsection{Remainder estimates}
We will show that $Y_\eps$ solves the equation  $$(H_\eps-\zeta)Y_\eps=h+r_\eps,$$
where the remainder term $r_\eps$ is small in $L_2$-norm uniformly with respect to $h$. Let us compute  $r_\eps$. If $|x|>\eps$, then we have
\begin{equation*}
  r_\eps(x)=\big( -\tfrac{d^2}{dx^2}+V(x)-\zeta\big)\big(y(x)
  +\rho_\eps(x)\big)-h(x)=-\rho_\eps''(x) +(V(x)-\zeta)\rho_\eps(x),
\end{equation*}
by \eqref{C1AsymptoticsEq}. If $|x|<\eps$, then
\begingroup
\allowdisplaybreaks
\begin{align*}
   r_\eps(x)  &=
     -\frac{d^2}{dx^2}\big(Y_\eps\xe\big)+(V(x)-\zeta) Y_\eps\xe\\
     &+\eps^{-3}\int\limits_{-\eps}^\eps\big( \bar{g}\se f\xe+\bar{f}\se g\xe\big)Y_\eps\se\,d\tau\,
     +\eps^{-1}q\xe Y_\eps\xe-h(x)\\
     &= \eps^{-2} \big(-u''\xe+(Qu)\xe\big)
       +\eps^{-1} \big(-v''\xe+(Qv)\xe+ q\xe u\xe\big)\\
     &  +\big(-\vartheta_\eps''\xe+(Q\vartheta_\eps)\xe-h(x)-\gamma_\eps m\xe\big)+q\xe v\xe+\gamma_\eps m\xe\\
     &+(V(x)-\zeta) Y_\eps\xe=q\xe v\xe+\gamma_\eps m\xe+(V(x)-\zeta) Y_\eps\xe,
\end{align*}%
\endgroup
by \eqref{bvpV0}, \eqref{bvpV1} and \eqref{C2ProblemVe1}. Hence
\begin{equation*}
  r_\eps(x)=
      \begin{cases}
        -\rho_\eps''(x)+(V(x)-\zeta)\rho_\eps(x), & \text{if } |x|>\eps,\\
        q\xe v\xe+\gamma_\eps m\xe+(V(x)-\zeta) Y_\eps\xe, & \text{if } |x|<\eps.
      \end{cases}
\end{equation*}
We will prove that $r_\eps$ is small in the $L_2(\Real)$-norm and
also show that the non-zero contribution in the norm $\|Y_\eps\|$, as $\eps\to 0$, is produced by the function $y$ only.

\begin{prop}\label{PropRepsSeps}
For all $h\in L_2(\Real)$ functions $r_\eps=(H_\eps-\zeta)Y_\eps-h$ and $s_\eps=Y_\eps-y$ satisfy the estimate
\begin{equation*}
    \|r_\eps\|+\|s_\eps\|\leq c \eps^{1/2} \|h\|,
\end{equation*}
where the constant $c$ does not depend on $h$ and $\eps$.
\end{prop}
\begin{proof}
  First we record some estimates on $y$, $u$, $v$ and $\vartheta_\eps$.
We  observe that $(\mathcal{H}-\zeta)^{-1}$ is a bounded operator from~$L_2(\Real)$ to the domain of $\mathcal{H}$ equipped with the
graph norm, and  the domain is a subspace of $W_{2,loc}^2(\Real\setminus \{0\})$. Therefore we have
\begin{equation*}
\|y\|_{W_2^2(-a,0)}+\|y\|_{W_2^2(0,a)}\leq c_1\|h\|
\end{equation*}
for any $a>0$, and thus
\begin{equation}\label{EstY}
    \|y\|_{C^1(-a,0)}+\|y\|_{C^1(0,a)}\leq c_2\|h\|,
\end{equation}
by the Sobolev embedding theorem.
 In particular, we  have
\begin{equation*}
|y_-|+|y_+|\leq c_3\|h\|.
\end{equation*}
It follows from \eqref{VInCaseKappaNotZero} and the last bound that
\begin{equation}\label{EsyV}
  \|u\|_{L_2(\cI)}\leq c_4\big(|y_-|+|y_+|\big)
  \leq c_5 \|h\|.
\end{equation}
Using the bound \eqref{NHPUestimate} along with \eqref{EsyV}, we estimate
\begin{equation}\label{EstV1}
  \|v\|_{W_2^2(\cI)}\leq c_6(|y_-|+|y_+|+\|qu\|_{L_2(\cI)})\leq c_7\|h\|,
\end{equation}
since the potential $q$ is bounded. To estimate $\vartheta_\eps$, we apply  \eqref{NHPUestimate} to problem \eqref{C2ProblemVe1}
\begin{equation}\label{EstVeps}
   \|\vartheta_\eps\|_{W_2^2(\cI)}\leq c_{8}(|\alpha_\eps(h)|+|\beta_\eps(h)|+|\gamma_\eps(h)|
   +\|h(\eps\cdot)\|_{L_2(\cI)})\leq c_{9}\,\eps^{-1/2} \|h\|,
\end{equation}
where we  used \eqref{EstAlphaBetaGamma}.
Hence  inequalities \eqref{EsyV},  \eqref{EstV1} and \eqref{EstVeps} provide the bound
\begin{multline}\label{EstHatY}
  \|Y_\eps(\eps^{-1}\cdot)\|_{L_2(-\eps,\eps)}= \eps^{1/2}\|Y_\eps\|_{L_2(\cI)}
  =\eps^{1/2}\|u+\eps v+\eps^2 \vartheta_\eps\|_{L_2(\cI)}
  \\
  \leq
  \eps^{1/2}\|u\|_{L_2(\cI)}+\eps^{3/2} \|v\|_{L_2(\cI)}+\eps^{5/2} \|\vartheta_\eps\|_{L_2(\cI)}\leq  c_{10}\eps^{1/2} \|h\|.
\end{multline}

In order to estimate $\rho_\eps$ we calculate the jumps of $w_\eps$. Recalling that $v(-1)=0$ and $\vartheta_\eps(-1)=0$ for all cases, we have
\begin{equation}\label{C2Jumps}
\begin{aligned}
  &[w_\eps]_{-\eps}=y_--y(-\eps),&& [w_\eps]_{\eps}=y(\eps)-y_++\eps v(1)+\eps^2\vartheta_\eps(1),\\
  &[w_\eps']_{-\eps}=y'_--y'(-\eps)+\eps \alpha_\eps,&& [w_\eps']_{\eps}=y'(\eps)-y'_+-\eps \beta_\eps.
\end{aligned}
\end{equation}
There exists a constant  being independent of $\eps$ and $y$ such that
\begin{equation}\label{EstY(Eps)-Y(0)}
\big|y^{(k)}(-\eps)-y^{(k)}_-\big|+
\big|y^{(k)}(\eps)-y^{(k)}_+\bigl|\leq C\eps^{1/2}\|h\|
\end{equation}
for $k=0,1$, since
\begin{equation*}
\bigr|y^{(k)}(\pm\eps)-y^{(k)}_\pm\bigl|\leq \left|\int_0^{\pm\eps}|y^{(k+1)}(x)|\,dx\right|
   \leq C\eps^{1/2} \|y\|_{W_2^2((-1,1)\setminus \{0\})}.
\end{equation*}
Then utilizing estimate \eqref{REst} in Proposition~\ref{PropW22Corrector}  (with $\rho_\eps$ and $w_\eps$ in place of $\rho$ and $w$, respectively) we obtain the bound
\begin{multline}\label{EstRhoEps}
  |\rho_\eps(x)|+|\rho_\eps''(x)|
  \leq
  c_{11}\Big(
  |y(-\eps)-y_-|+ |y(\eps)-y_+|+|y'(-\eps) -y'_-|\\+|y'(\eps)-y'_+|
  +\eps \big(|v(1)|+|\alpha_\eps|+|\beta_\eps|\big)+\eps^2|\vartheta_\eps(1)|
  \Big)
  \leq
  c_{12}\eps^{1/2}\|h\|
\end{multline}
for $|x|\geq \eps$, in view of \eqref{EstAlphaBetaGamma}, \eqref{EstV1}, \eqref{EstVeps} and \eqref{EstY(Eps)-Y(0)}.
Using this bounds along with \eqref{EstHatY}, we estimate
\begin{multline*}
  \|r_\eps\|\leq  c_{13} \big(\|\rho_\eps''+(V-\zeta)  \rho_\eps\|\\+\|q(\eps^{-1}\cdot) v(\eps^{-1}\cdot)+\gamma_\eps m(\eps^{-1}\cdot)+(V-\zeta)  Y_\eps(\eps^{-1}\cdot)\|_{L_2(-\eps,\eps)}\big)\\
  \leq c_{14}\max\limits_{|x|>\eps}(|\rho_\eps|+|\rho_\eps''|)+c_{15} \eps^{1/2}\big(\|v\|_{L_2(\cI)}+\|Y_\eps\|_{L_2(\cI)}\big)\leq c_{16}\eps^{1/2}\|h\|
\end{multline*}
as desired.
We still have to  estimate
\begin{equation*}
 s_\eps(x)=
  \begin{cases}
      \rho_\eps(x)   & \text{if }|x|>\eps,\\
      u\xe+\eps v\xe+\eps^2 \vartheta_\eps\xe-y(x)& \text{if }  |x|<\eps.
  \end{cases}
\end{equation*}
We can as before invoke \eqref{EstY}, \eqref{EstHatY} and \eqref{EstRhoEps} to derive the bound
\begin{multline*}
   \|s_\eps\|\leq c_1\big(\|\rho_\eps\|+\|Y_\eps(\eps^{-1}\cdot)\|_{L_2(-\eps,\eps)}+\|y\|_{L_2(-\eps,\eps)}\big)\\
   \leq c_2\eps^{1/2}(\|h\|+\max_{|x|\leq \eps}|y(x)|)\leq c_3\eps^{1/2}\|h\|,
\end{multline*}
which completes the proof of the proposition.
\end{proof}

\subsection{End of the proof}
Recall that $y_\eps=(H_\eps-\zeta)^{-1}h$ and $y=(\mathcal{H}-\zeta)^{-1}h$ for given $h\in L_2(\Real)$ and a complex number $\zeta$ with non-zero imaginary part.
By definition of $r_\eps$ and $s_\eps$ we have $(H_\eps-\zeta)Y_\eps=h+r_\eps$ and $Y_\eps=(\mathcal{H}-\zeta)^{-1}h+s_\eps$. We conclude from this that
$(H_\eps-\zeta)^{-1}h=Y_\eps-(H_\eps-\zeta)^{-1}r_\eps$ and $(\mathcal{H}-\zeta)^{-1}h=Y_\eps-s_\eps$,
hence that
\begin{multline*}
    \|(H_\eps-\zeta)^{-1}h-(\mathcal{H}-\zeta)^{-1}h\|=
    \|s_\eps-(H_\eps-\zeta)^{-1}r_\eps\|\\
    \leq\|s_\eps\|+ \|(H_\eps-\zeta)^{-1}\|\, \|r_\eps\|
    \leq\|s_\eps\|+ |\Im\zeta|^{-1}\|r_\eps\|  \leq C \eps^{1/2}\|h\|,
\end{multline*}
in view of Proposition~\ref{PropRepsSeps}. The last bound establishes
the norm resolvent convergence of  $H_\eps$ to the operator $\mathcal{H}$, which is the desired conclusion.

\section{Proof of Theorem~\ref{ThmS}}

\subsection{Case B1}
We begin from  the case in which operator $B$ possesses two linearly independent half-bound states.
We  assume that $f_0=g_0=0$, $\pi =0$ and $\kappa\neq 0$. But suppose now instead of $a_2\neq\overline{\kappa} a_1$,  as in the case A1, that the equality  $a_2=\overline{\kappa} a_1$ holds (see the graph in Fig.~\ref{FigBranchingAlgorithm}). Starting the proof as in \ref{Subsect41}, we look for uniform approximation $Y_\eps$ in the form \eqref{YepsA1KappaNot0} to a solution of  equation $(H_\eps-\zeta)y_\eps=h$.
In this case, we first see the difference  in matrix condition \eqref{MatrixEquality}, because $a_2=\overline{\kappa} a_1$ and therefore the matrix on the left is now degenerate.
Since $\kappa a_2=|\kappa|^2 a_1$, \eqref{MatrixEquality} can be written in the form
\begin{equation*}
  \begin{pmatrix}
    \frac{a_2}{|\kappa|^2} & -1\\
    \frac{a_2}{|\kappa|^2} & -1\\
  \end{pmatrix}
  \begin{pmatrix}
    y_+\\ y'_+
  \end{pmatrix}=
  \begin{pmatrix}
    \frac{a_2}{|\kappa|^2}-a_0 & -1\\
    0 & 0\\
  \end{pmatrix}
  \begin{pmatrix}
    y_-\\ y'_-
  \end{pmatrix}.
\end{equation*}
It follows immediately from this that
\begin{equation}\label{CaseB1conds}
y'_--\left(|\kappa|^{-2}a_2-a_0\right)y_-=0, \qquad
y'_+-|\kappa|^{-2}a_2y_+=0.
\end{equation}
Hence we introduce the limit operator $\mathcal{H}$ as the  Schr\"{o}dinger operator on the line acting via $\mathcal{H}\psi=-\psi''+V\psi$ on the domain
\begin{equation*}
 \dom \mathcal{H}=\left\{ \psi \in \mathcal{V} \colon \psi'(-0)=\theta_1 \psi(-0), \quad \psi'(+0)=\theta_2 \psi(+0)\right\},
\end{equation*}
where $\theta_1=|\kappa|^{-2}a_2-a_0$ and $\theta_2=|\kappa|^{-2}a_2$.
Thus $\mathcal{H}=\mathcal{R}_-(\theta_1)\oplus \mathcal{R}_+(\theta_2)$.

Turning to approximation $Y_\eps$, we assume that  $y=(\mathcal{H}-\zeta)^{-1}h$ is a $L_2(\Real)$-function solving  the equation $-y''+(V-\zeta) y=h$, subject to  coupling conditions \eqref{CaseB1conds}.
Next, $u$ is  a half-bound state given by \eqref{VInCaseKappaNotZero}, $v$ and $\vartheta_\eps$ are solutions to problems \eqref{bvpV1} and \eqref{C2ProblemVe1} respectively such that $v(\pm1)=0$ and $\vartheta_\eps(\pm1)=0$, by Proposition~\ref{PropSolvabilityOfNHP}. The jumps $[w_\eps]_{\pm\eps}$ and $[w_\eps']_{\pm\eps}$ given by \eqref{C2Jumps} are small as $\eps\to 0$ uniformly on $h\in L_2(\Real)$. Hence there exists a small corrector $\rho_\eps$ satisfying estimate \eqref{EstRhoEps} such that   $Y_\eps\in \dom H_\eps$. In addition, $Y_\eps$ satisfies the equation $(H_\eps-\zeta)Y_\eps=h+r_\eps$ with the remainder $r_\eps$ that can be estimated as in Proposition~\ref{PropRepsSeps}.
By the argument used at the  end of the proof of Theorem~\ref{ThmC}, we show that  $H_\eps$ converge to $\mathcal{R}_-(\theta_1)\oplus \mathcal{R}_+(\theta_2)$ as $\eps\to 0$ in the norm resolvent sense.

\subsection{Case B2} Suppose that   $\lambda=0$, $f_0g_0\neq 0$, $f_0g_1\neq f_1g_0$ and  $\sigma_- \sigma_+=0$. According to Lemma~\ref{LemmaHBS} \textit{(i)}, operator  $B$ possesses  half-bound state $\sigma$.
Looking for  approximation $Y_\eps$, we  set $u=c_0\sigma$ with some constant $c_0$. As in the case A2, we have
\begin{equation}\label{y=cSigma}
 y_-=c_0\sigma_-, \quad y_+=c_0\sigma_+, \quad \bar{\sigma}_+y'_+-\bar{\sigma}_-y'_-=c_0\sigma_*.
\end{equation}
In view of Proposition~\ref{PropSigmaPrps} \textit{(ii)} only one of the values $\sigma_-$ and $\sigma_+$ is equal to zero. Assume for instance $\sigma_-=0$.
Then we immediately deduce $y_-=0$, $c_0=\sigma_+^{-1}y_+$ and therefore
$y'_+=\sigma_+^{-2}\sigma_*y_+$. Hence in this case the limit operator
$\mathcal{H}$ as $\eps\to 0$ is the direct sum $\mathcal{D}_-\oplus \mathcal{R}_+(\theta)$, where $\theta=\sigma_+^{-2}\sigma_*$.

Let now  $y=(\mathcal{H}-\zeta)^{-1}h$.  Since a solution $v$ of \eqref{bvpV1} is defined up to term $c_1\sigma$ and $\sigma(1)=\sigma_+\neq 0$, we can find a unique solution $v$ such that $v(1)=0$.
We also assume that $\vartheta_\eps$ solves the  problem
\begin{equation*}
  -\vartheta_\eps ''+Q \vartheta_\eps=h(\eps \,\cdot),\quad t\in \cI,\qquad
   \vartheta'_\eps(-1)=0,\quad \vartheta'_\eps(1)=\beta_\eps,
\end{equation*}
where $\beta_\eps=-\sigma_+^{-1}\,(\sigma, h(\eps \,\cdot))$.
The problem admits a solution such that $\vartheta_\eps(1)=0$.
Thus  we built approximation $Y_\eps\in\dom H_\eps$ and
the rest of the proof is word for word as in the proof of the previous theorem. The subcase $\sigma_+=0$ is treated similarly.

\subsection{Case B3}
This case collects all the subcases, in which the limit operator is the direct sum $\mathcal{D}_-\oplus \mathcal{D}_+$ of the unperturbed half-line Schr\"{o}dinger operators with potential $V$,  subject to the Dirichlet boundary condition at the origin.
In fact, if $\lambda\neq 0$, then
operator $B$ has no zero-energy resonance, i.e., problem \eqref{bvpV0} admits a trivial solution $u=0$ only. In view of coupling conditions \eqref{CouplingYV0}, it immediately follows that $y_-=0$ and $y_+=0$.
If  $f_0g_0\neq 0$, $f_0g_1=f_1g_0$, $\sigma_-=0$ and $\sigma_+=0$, then \eqref{y=cSigma} also implies  $y_-=0$ and $y_+=0$. Finally, in the case $f_0=0$, $g_0=0$, $\pi=0$, $\kappa=0$, $a_2=0$ and $a_1\neq0$, it follows from the second condition in \eqref{SolvCondsKappaZero} that $y(0)=0$.
The same proof, as in the previous cases, works in the case B3.

\medskip

The proof of Theorem~\ref{Thm3} is actually contained in the proofs of Theorem~\ref{ThmC} and Theorem~\ref{ThmS}. Estimate \eqref{MainEst} immediately follows from Proposition~\ref{PropRepsSeps}. The order of convergence is optimal, because the estimate   $\|h(\eps\cdot)\|_{L_2(\cI)}\leq c\,\eps^{-1/2} \|h\|$
in the proof of Proposition~\ref{PropRepsSeps} is precise and cannot be improved for $L_2$-functions.

\bigskip

\emph{Acknowledgements.}
I thank the anonymous Referee for careful reading of the manuscript and
valuable remarks and suggestions; the  first version of the paper was significantly improved by these suggestions.

\bibliographystyle{amsplain}

\end{document}